\documentclass{ws-jnmp}

\usepackage[T1]{fontenc}
\usepackage{graphicx}

\usepackage{color}
\definecolor{MyLinkColor}{rgb}{0,0,0.4}

\newcommand{\0}{\Omega}
\newcommand{\e}{\varepsilon}
\newcommand{\p}{\partial}
\newcommand{\wt}{\widetilde}
\newcommand{\ov}{\overline}

\newcommand{\G}{\Gamma}

\newcommand{\R}{\mathbb{R}}

\begin{document}

\newtheorem{thm}{Theorem}[section]
\newtheorem{prop}[thm]{Proposition}
\newtheorem{cor}[thm]{Corollary}

\arttype{Article} 

\markboth{A.--V. Matioc and B.--V. Matioc}
{On periodic water waves with Coriolis effects and isobaric streamlines}

%
%
\copyrightauthor{}

\title{{\sc On periodic water waves with Coriolis effects and isobaric streamlines}}

\author{\footnotesize ANCA--VOICHITA MATIOC}

\address{Faculty of Mathematics, University of Vienna, \\ Nordbergstra\ss e 15, 1090 Vienna, Austria\\
\email{anca.matioc@univie.ac.at}}

\author{\footnotesize BOGDAN--VASILE MATIOC}

\address{Faculty of Mathematics, University of Vienna, \\ Nordbergstra\ss e 15, 1090 Vienna, Austria\\
\email{bogdan-vasile.matioc@univie.ac.at}}

\maketitle

\begin{abstract}
In this paper we prove that  solutions of the {\em f}--plane approximation for equatorial geophysical deep water waves, which have the property that the pressure is constant along the streamlines and do not possess
stagnation points,  are Gerstner-type waves.
Furthermore, for waves traveling over a flat bed, we prove that there are only laminar flow  solutions with these properties.
\end{abstract}


\ccode{2010 Mathematics Subject Classification: Primary: 76B15; Secondary: 37N10.}

\section{Introduction}

The motion of a fluid layer located on the Earth's surface is also influenced by Earth's rotation around the polar axis.
For fluid motions localized near the Equator, the variation of the Coriolis parameter
may be neglected and geophysical water waves in this region are modeled by the so-called {\em f}--plane approximation. 
The physical relevance of the {\em f}--plane approximation for geophysical waves is discussed in \cite{AC12}.

In this paper we consider periodic solutions  of the {\em f}--plane approximation which possess isobaric streamlines, that is the pressure is constant along the streamlines of the flow.
Furthermore, we assume that the wave speed exceeds the horizontal  velocity of all particles in the fluid so that   stagnation points are excluded.
Based upon  a priori properties for such waves, which we establish herein, and regularity results for quasilinear elliptic equations, cf. \cite{Mo66},  
we prove that, if the ocean depth is infinite, such waves have an explicit Lagrangian description.
These solutions were found  initially by Gerstner in \cite{Ge09}  and rediscovered later on by Rankine \cite{ra63} in the context of  flows without Coriolis effects. 
More recently, their properties have been analyzed in 
\cite{Co01, Con11, He08}.
They may be adapted to describe edge waves in homogeneous \cite{C01a} and   stratified fluids \cite{Ra11, Yih66}, or 
 gravity waves solutions of  the {\em f}--plane approximation \cite{AM12xx}.

A important characteristic of Gerstner's  solutions  is  that they describe rotational flows with the fluid particles moving on circles, a feature which does not hold  for irrotational periodic deep water waves as seen in 
\cite{CoEhVi08, Am12} by means of linear theory.
For the description of the particle paths in linear  waves traveling over a flat bottom we refer to \cite{CoVi08, ioK08, MaA10}.
We enhance that the results on the particle trajectories for linear water waves  have nonlinear counterparts for  the  governing equations, cf. \cite{Co06,CoSt10, He06, He08-1}.  

In the context of waves in water over a flat bed, we prove that any solution of 
the {\em f}--plane approximation which has isobaric streamlines and no stagnation points  must be a laminar flow, that is the 
 streamlines are straight lines.

Though the angular speed of Earth's rotation is a well determined constant, our analysis remains valid for any arbitrary value of the angular  speed.
Particularly, we recover  previous results established for gravity water waves without Coriolis effects \cite{Ka04}, which are a special case of the situation analyzed herein (zero angular speed). 
It is worth mentioning that the symmetry of the Gerstner-type waves reflects the one obtained for gravity waves in water of finite or infinite depth, cf. \cite{ CoEhWa07, CoEs04_b,Co-Es04_1, Mama12, OS01}.

The outline of the paper is as follows: 
we present in Section \ref{S1} the mathematical formulation of the problem we deal with and  state the main results Theorems \ref{MT:1} and \ref{MT:2}. 
To this end we introduce in a Lagrangian framework  Gerstner  solutions for the deep-water problem and discuss their properties.
In Section \ref{S2} we reformulate the problem and establish some preliminary properties common for both deep and finite depth waves with isobaric streamlines and  without stagnation points. 
 Section \ref{S3} is dedicated to the proof of Theorem \ref{MT:1}, while in Section \ref{S4} we prove our second main result Theorem \ref{MT:2}.

\section{The mathematical model and the main results}\label{S1}
We consider herein a rotating frame with the origin at a point on Earth's surface,  the $X-$axis being chosen horizontally due east, the $Y-$axis horizontally  due north, and the $Z-$axis pointing  upward. 
Furthermore, we let  $Z=\eta(t,X,Y)$ be the upper free boundary of a fluid layer which may have a finite depth, the plane $Z=-d$, $d\in\R,$ being the impermeable bottom of ocean,  or is unbounded and in this situation we deal with deep water waves.
In the fluid layer located near the Equator, the governing equations in the $f-$plane approximation are, cf. \cite{GalStR07}, the Euler equations
\begin{equation}\label{Euler}
\left\{
\begin{array}{rllllll}
u_t+uu_X+vu_Y+wu_Z+2\omega w&=&-P_X/\rho \\
v_t+uv_X+vv_Y+wv_Z&=& -P_Y/\rho,\\
w_t+uw_X+vw_Y+ww_Z-2\omega u&=&-P_Z/\rho -g,
\end{array}
\right.
\end{equation}
and the equation of conservation of mass
\begin{equation}\label{maco}
u_X+v_Y+w_Z=0.
\end{equation}
Here $t$ is the time variable, $(u,v,w)$  the fluid velocity, $\omega=73\cdot 10^{-6} rad/s$ is the  rotation speed of the Earth round the polar axis
towards east, $\rho$ is the density constant of the water, $g=9,8 m/s^2$ is the  gravitational acceleration  at the Earth's surface, and $P$ is the pressure.  

At the wave surface, the pressure of the fluid matches the  atmospheric pressure $P_0:$
\begin{equation}\label{Dbc}
P=P_0 \quad \quad \text{on} \quad Z=\eta(t,X,Y).
\end{equation}
Moreover, the free surface of the wave consists at each moment of the same fluid particles, so that we obtain  the kinematic
boundary condition
\begin{equation}\label{kbc}
w=\eta_t+u\eta_X \quad \quad \text{on} \quad Z=\eta(t,X,Y).
\end{equation}
Since for finite depth waves the bottom of the ocean is assumed to be impermeable, we impose the no-flux condition
\begin{equation}\label{B1}
w=0 \quad \quad \text{on} \quad Z=-d.
\end{equation}
For deep water waves we assume that
\begin{equation}\label{B2}
(u,w)\to 0 \quad \quad \text{for $ Z \to-\infty$ uniformly in $(t,X)$,}
\end{equation}
meaning that at great depths there is practically no flow.

In this paper we consider traveling waves, with the velocity field, the pressure, and the free surface exhibiting an $(t,X)-$dependence of the form $(X-ct),$ where $c>0$ is the speed of the  wave surface.  
Moreover, we seek two-dimensional flows, independent upon the $Y-$coordinate and with $v \equiv 0$ throughout the flow.
Introducing the new variables
\begin{equation}\label{E:var}
x:=X-ct\qquad\text{and}\qquad z:=Z,
\end{equation}
the governing equations for water waves  reduce to  the following nonlinear free-boundary problem
\begin{equation}\label{E:T}
\left\{
\begin{array}{rlll}
(u-c)u_x+2\omega w&=&-P_x/\rho & \qquad \text{in $  \0_\eta,$}\\
(u-c)w_x-2\omega u&=&-P_z/\rho -g  &\qquad \text{in $  \0_\eta,$}\\
u_x+w_z &=& 0  & \qquad\text{in $ \0_\eta,$}\\
P&=&P_0 & \qquad\text{on $  z=\eta(x),$}\\
w&=&(u-c)\eta_x & \qquad \text{on $ z=\eta(x),$}
\end{array}
\right.
\end{equation}
supplemented by the boundary condition
\begin{equation}\label{E:T1}
\begin{array}{rllllllll}
w &=& 0  &\qquad  \text{on $ z =-d$}
\end{array}
\end{equation}
for   waves traveling over a flat bed, respectively
\begin{equation}\label{E:T2}
\begin{array}{rllllllll}
(u,w) \to 0  & \qquad  \text{for $ z \to-\infty$ uniformly in $x$}
\end{array}
\end{equation}
in the infinite-depth case.
The fluid domain $\0_\eta$ is bounded from above by the graph of $\eta$ and is unbounded from below in the infinite depth case, respectively is bounded by the line $z=-d$ when the ocean depth is  finite.  

The solutions we consider  are periodic in the variable $x$, that is $u,w,P,\eta$ are all periodic in $x$ with the same period, and  have no stagnation points. 
The latter property is satisfied if we assume that
\begin{equation}\label{COND1}
\sup_{\0_\eta} u<c.
\end{equation}
Moreover, we require a priori that the solutions have  the following regularity:
 \begin{equation}\label{COND2}
\text{$\eta\in C^{3}(\R)$ and $(u,w,P)\in \left(C\,^2( \ov \0_\eta)\right)^3$.}
\end{equation}

Our first main result is the following theorem.
\begin{thm}\label{MT:1}
 Assume that $(u,v,P,\eta)$ is a solution of \eqref{E:T}-\eqref{E:T1} which satisfies the relations \eqref{COND1}-\eqref{COND2}.
 If the  pressure is constant along  the streamlines, then the  flow is laminar.
\end{thm}

Besides laminar flow solutions, which are characterized by the fact that $\eta$ is constant and $(u,v,P)$ depend only upon the $z$ coordinate,
in the infinite depth case there is a family of explicit traveling wave solutions of the problem \eqref{E:T} and \eqref{E:T2}, which is due to Gerstner.
To present these special solutions we adopt a Lagrangian framework and describe the trajectories of each particle in the fluid:
\begin{equation}\label{Lag}
 (X(t, a, b), Z(t, a, b)):=\left(a-\frac{e^{kb}}{k}\sin(k(a-ct)),h_0+b+\frac{e^{kb}}{k}\cos(k(a-ct))\right)
\end{equation}
 for all  $a\in\R$, $b\leq b_0,$ and $t\geq0.$ 
Hereby $k>0$, $b_0\leq 0,$ $h_0\in\R$ is arbitrary, and   $c$ is the speed of the wave.
Each particle within the fluid is uniquely determined by a pair  $(a,b)\in\R\times (-\infty,b_0).$ 
The curve parametrized by $(X(t, \cdot, b_0), Z(t, \cdot, b_0))$ is the profile of the wave and is a  trochoid when $b_0<0$, respectively a cycloid when $b_0=0$ (see the figures in Chapter 4 of \cite{Con11}).
The latter curve has upward cusps, so that the waves cannot be extended for values of $b\geq0.$
 The  speed $c$ of the wave is given by
\[
c=\frac{-\omega+\sqrt{\omega^2+gk}}{k },
\]
 cf. \cite{AM12xx}. Our second main result is the following theorem.
\begin{thm}\label{MT:2}
 Assume that $(u,v,P,\eta)$ is a solution of \eqref{E:T} and \eqref{E:T2} which satisfies \eqref{COND1} and \eqref{COND2}.
If the pressure is constant along each  streamline, then  $(u,v,P,\eta)$ is one of the solutions described by \eqref{Lag}.
\end{thm}

\section{Equivalent formulations and a priori properties of solutions}\label{S2}
In order to investigate the steady flow problems for finite depth and deep water waves we find  equivalent formulations which are more suitable to handle.
First, we reformulate the problem in terms of a   stream function $\psi$.
To deal with both finite and infinite depth  cases at once, we define the stream function by the following relation
\[
\psi(x,z):=-\int_{z}^{\eta(x)}(u(x,s)-c)\, ds \qquad\text{for $(x,z)\in\ov\0_\eta.$}
\]
Then, it follows by direct computations that $\nabla\psi=(-w,u-c).$
Consequently, the streamlines of the steady flow, which coincide with the particle paths,  are the  level curves of $\psi.$
Indeed, if the curve $(x(t),z(t))$ describes the motion of a fluid particle, that is $(x'(t),z'(t))=(u(x(t),z(t))-c,w(x(t),z(t)))$, 
then
\[
\frac{d}{dt} \psi(x(t),z(t))=\left\langle\nabla \psi(x(t),z(t))\big|(u(x(t),z(t))-c,w(x(t),z(t)))\right\rangle=0 \qquad\text{for all $t\geq0$.}
\] 
Furthermore, using the implicit function theorem and  \eqref{COND1}  we see that  the level   curves of $\psi$ are in fact graphs of periodic functions defined on the entire  real line.
Additionally, we compute that
\[
\Delta \psi=\psi_{xx}+\psi_{zz}=u_z-w_x=\gamma\qquad\text{in $\0_\eta$,}
\]
the function $\gamma=\gamma(x,z)$ being the vorticity of the flow.
By \eqref{E:T1}, for   waves over a flat bed $\psi_x=-w=0$ on $x=-d,$ and we deduce that there is a positive constant $p_0$ such that $\psi=p_0$ on $z=-d.$
Also, observe that $\psi=0$ on $z=\eta(x).$
 
Condition \eqref{COND1} enables us  also  to introduce new variables by means of  the  hodograph transformation $H:\0_\eta\to\0$,
defined by
\[
H(x,z):=(q,p)(x,z):=(x,-\psi(x,z)),\qquad (x,z)\in\ov \0_\eta,
\]
whereby $\Omega:=\{(q,p): -p_0<p<0 \}$  and $\Omega:=\{(q,p): p<0 \}$ for the finite and infinite depth case, respectively. 
The mapping $H$ is a diffeomorphism and the following relations are satisfied 
\begin{align*}
  \left( \begin{array}{cccc}
q_x&q_z \\
p_x&p_z
 \end{array} \right)= \begin{pmatrix}
1 \,  & \, 0\\
w \, & \, c-u
 \end{pmatrix}   
\qquad\text{and}\qquad 
\left( \begin{array}{cccc}
x_q&x_p \\
z_q&z_p
 \end{array} \right)=\left( \begin{array}{ccc}
1 \, & \, 0\\[1ex]
\displaystyle\frac{w}{u-c} \, & \,  \displaystyle\frac{1}{c-u} 
 \end{array} \right).
\end{align*}
A simple computation shows now that $\p_q(\gamma\circ H^{-1})=0$ in $\0,$ meaning that there exists a continuously differentiable  function $\gamma=\gamma(p)$ such that
$\gamma(x,z)=\gamma(-\psi(x,z))$ for all $(x,z)\in\0_\eta.$  
Finally, defining the  hydraulic head  by the expression
\begin{equation*}
E:=\frac{(u-c)^2+w^2}{2}+(g-2\omega c)z+\frac{P}{\rho}-2\omega \psi+\int_0^{-\psi}\gamma(s)\, ds\qquad \text{in $\0_\eta$,}
\end{equation*}
one can easily show  that there exists a constant $C$ such that  $E=C$ in $\0_\eta.$
Introducing the primitive $\Gamma$ of $\gamma$ by the relation
\[
\Gamma(p):=\int_0^p\gamma(s)\,ds-C,
\]
we find that $(\eta,\psi)$ solves the following problem: 
 \begin{equation}\label{S:1}
\left\{
\begin{array}{rllllllll}
\Delta \psi&=&\gamma(-\psi) &\qquad \text{in $ \Omega_\eta,$}\\
|\nabla\psi|^2/2+(g-2\omega c)z+P/\rho-2\omega \psi+\Gamma&=&0  &\qquad \text{in $ \Omega_\eta$},
\end{array}
\right.
\end{equation}
supplemented by 
\begin{equation}\label{S:2}
\begin{array}{rllllllll}
\psi&=&0  &\qquad \text{on $ z=-d$}
\end{array}
\end{equation}
for  waves traveling over a flat bed, respectively  
\begin{equation}\label{S:3}
\begin{array}{rllllllll}
\nabla \psi&\to&(0,-c)  &\qquad \text{for $ z\to-\infty$ uniformly in $x$}
\end{array}
\end{equation}
for deep water waves.
In fact, it is not difficult to show that problems \eqref{E:T}-\eqref{E:T1} and \eqref{S:1}-\eqref{S:2}  (resp. $\{\eqref{E:T},\eqref{E:T2}\}$ and $\{\eqref{S:1}, \eqref{S:3}\}$) 
are equivalent in the sense that each solution of the first problem corresponds  
to  a unique solution of the second one. 
Since we consider only  waves having the property that the pressure is constant along the streamlines we 
 find a function $P\in C^2([-p_0,0])$  in the finite depth case (resp. $P\in C^2((-\infty,0]) $  in the infinite depth case)
with the  property that $P(x,z)=P(-\psi(x,z))$ for all $(x,z)\in\ov \0_\eta.$

To obtain a second equivalent formulation of the original problems, we introduce the height function $h:\0\to\R$ by the relation
\[
h(q,p)=z\qquad\text{for $(q,p)\in\0.$}
\]
It follows readily from the definition of $h$ and of the coordinates transformation $H$ that $h$ solves the following equations 
\begin{equation}\label{R:1}
\left\{
\begin{array}{rllll}
(1+h_q^2)h_{pp}-2h_ph_qh_{pq}+h_p^2h_{qq}-\Gamma' h_p^3&=&0& \qquad\text{in $\0$},\\[1ex]
\displaystyle\frac{1+h_q^2}{2h_p^2}+(g-2\omega c)h+\displaystyle\frac{P}{\rho}+2\omega p+\Gamma&=&0&\qquad\text{in $\0$},
\end{array}
\right.
\end{equation}
and 
\begin{equation}\label{R:2}
\begin{array}{rllll}
h&=&-d&\qquad\text{on $p=-p_0$}
\end{array}
\end{equation}
for waves traveling over a flat bed, respectively
\begin{equation}\label{R:3}
\begin{array}{rllll}
\nabla h&\to &(0,1/c)&\qquad\text{on $p\to-\infty$ uniformly in $q$,}
\end{array}
\end{equation}
in the infinite depth case. 
The problems \eqref{R:1}-\eqref{R:2} and \eqref{S:1}-\eqref{S:2}  (resp. $\{\eqref{R:1}, \eqref{R:3}\}$ and $\{\eqref{S:1}, \eqref{S:3} \}$) 
are equivalent in the same sense defined before.
We note that \eqref{COND1} becomes
\begin{equation}\label{COND1'}
\inf_{\0} h_p>0,
\end{equation}
while, due to the fact that $\psi(q,h(q,p))=-p$ for $(q,p)\in\0$, we see that the streamlines of the steady flow are parametrized by the functions $h(\cdot,p).$ 

To simplify our notations, we set
\begin{equation}\label{notations}
\alpha:=g- 2 \omega c\qquad\text{and}\qquad Q(p):=\frac{P(p)}{\rho}+2\omega p \qquad\text{for all $p$ with $(0,p)\in\ov\0.$}  
\end{equation}

We establish now some properties which are a priori satisfied by the solutions of the system  \eqref{E:T}. 
\begin{lemma}\label{L:1}
 Let $(u,w, \eta, P)$ be a periodic solution of \eqref{E:T} satisfying \eqref{COND1} and \eqref{COND2}, and assume that the pressure $P$ is constant along the streamlines.
Then, there exists a continuously differentiable function $\beta=\beta(p)$ such that the corresponding height function $h$ satisfies 
\begin{align}\label{has}
 \frac{1}{h_p(q,p)}=Q'(p)h(q,p)+\beta(p) \qquad\text{for all $(q,p)\in\0$.}
\end{align}
\end{lemma}
\begin{proof}
 First, we note that the height function $h$ solves \eqref{R:1}. 
Differentiating the second  equation of \eqref{R:1} with respect to $p$ and $q,$ respectively, we obtain in $\0$ the following relations
\begin{align}\label{wp1}
 \frac{h_ph_qh_{pq}-(1+h_q^2)h_{pp}}{h_p^3}+\alpha h_p+Q'+\Gamma'=0.
\end{align}
and
\begin{align}\label{wp2}
 \frac{h_ph_qh_{qq}-(1+h_q^2)h_{pq}}{h_p^3}+\alpha h_q=0  
\end{align}
 We build the sum of  \eqref{wp1} and the first equation in \eqref{R:1} to find that
\begin{align}\label{wp3}
 \frac{h_p^2h_{qq}-h_p h_q h_{pq}}{h_p^3}+\alpha h_p+Q'=0\qquad\text{in $\0$.}
\end{align}
Next, we multiply \eqref{wp3} by $h_q$, \eqref{wp2} by $h_p$ and subtracting these new identities  we get
\[
 \frac{h_{pq}}{h_p^2}+Q'h_q=0,
\]
or equivalently
\[
\frac{\p}{\p q} \left( \frac{1}{h_p}- Q'h_q\right)=0.
\]
This yields the desired assertion \eqref{has}.
\end{proof}

As a further result we prove the following lemma, which enables us later on to identify  the characteristics of the flow.
\begin{lemma}\label{L:2} Under the same assumptions as in Lemma \ref{L:1},  the height function $h$, corresponding to a solution of \eqref{E:T}, satisfies 
\begin{align}\label{rela}
 a_3(p)h^3+a_2(p)h^2+a_1(p)h+a_0(p)=0 \qquad\text{in $\0$,}
\end{align}
whereby  
\begin{equation}\label{ais}
\begin{array}{lllllllll}
 a_0(p)&:=&-\alpha \beta+2\beta\beta'(Q+\Gamma)-Q' \beta^2-\beta^2(Q'+\G'),\\[1ex]
a_1(p)&:=&-\alpha Q'-2Q' \beta(Q'+\G')+2(Q+\G)(Q''\beta+Q'\beta'),
-2Q'^2\beta+2\alpha\beta\beta',\\[1ex]
a_2(p)&:=&-Q'^3-(Q'+\G')Q'^2+2\alpha(Q''\beta+Q'\beta']+2Q'Q''(Q+\G),\\[1ex]
a_3(p)&:=&2\alpha Q' Q''
\end{array}
\end{equation}
for all $p$ with $(0,p)\in\0.$
\end{lemma}
\begin{proof}
We differentiate first \eqref{has} with respect to $q$ and multiply the relation we obtain by $h_q$ to find that
\begin{align*}
 h_qh_{pq}=-Q' h_q^2h_p^2=\frac{Q'}{(Q' h+\beta)^2}+\frac{2(\alpha h+Q+\G)Q'}{(Q'h+\beta)^4}\qquad\text{in $\0$,}
\end{align*}
the last identity being a consequence of the second relation in \eqref{R:1} and \eqref{has}.
On the other hand, the latter relations yield
\begin{align*}
 h_q^2=-2\frac{\alpha h+Q+\G}{(Q' h+\beta)^2}-1,
\end{align*}
and differentiating this expression with respect to $p$ we arrive at
\begin{align*}
 h_qh_{pq}=-\frac{\alpha+(Q'+\G')(Q' h+\beta)}{(Q' h+\beta)^3}+2\frac{\alpha h +Q+\G}{(Q' h+\beta)^4}\left[Q'+(Q''h+\beta')(Q'h+\beta)\right].
\end{align*}
Identifying the two expressions we found for $h_qh_{pq}$, we obtain the desired relations. 
\end{proof}

We are interested here to determine the solutions of \eqref{E:T}-\eqref{E:T1} (resp. \eqref{E:T} and \eqref{E:T2}) which are not laminar (that is $h_q$ does not vanish in $\0$).
For these solutions we obtain the following restriction on the wave speed $c$.
\begin{lemma}\label{L:3} Assume that $(u,w, \eta, P)$ is a periodic solution of \eqref{E:T} satisfying \eqref{COND1} and \eqref{COND2}. 
Furthermore, we assume that the pressure $P$ is constant along the streamlines and the flow is not laminar, that is $h_q\not\equiv0$ in $\0.$ 
Then, we have that 
 $\alpha=g-2\omega c\neq0.$  
\end{lemma}
\begin{proof} Let us assume by contradiction that $\alpha=0.$
Then, the second relation of \eqref{R:1} is equivalent to
\[
(1+h_q^2)(Q'h+\beta)^2+2(Q+\G)=0\qquad\text{in $\0$}.
\]
We differentiate this relation with respect to $q$ and, since  $Q'h+\beta>0$, cf. \eqref{COND1'}, we arrive at 
\[
h_qh_{qq}(Q'h+\beta)+h_q(1+h_q^2)Q'=0\qquad\text{in $\0$}.
\]
We fix now $p<0 $  such that $h(\cdot,p)$ is not a constant function.
Since $h(\cdot,p)$ is a real analytic function, cf. \cite{AC11, BM11}, we find that
\[
h_{qq}(q,p)=-Q'(p)\frac{1+h_q^2(q,p)}{Q'(p)h(q,p)+\beta(p)}\qquad \text{ for all $q\in\R$.}
\]
If $Q'(p)=0,$ then $h_{qq}(q,p)=0$ for all $q\in\R, $ meaning that $h_q(q,p)=0$ for all $q\in\R $ and contradicting our assumption.  
On the other hand, if $Q'(p) \neq0$, then $h_{qq}(\cdot, p)$ has the sign of $-Q'(p)$ on the whole line, in contradiction with the periodicity of $h_q(\cdot,p_0).$
In conclusion, our assumption was false and the proof is complete.
\end{proof}

\section{Proof of Theorem \ref{MT:1}}\label{S3}
The proof of Theorem \ref{MT:1} follows by contradiction.
Assume thus that there exists a tuple  $(u,v,P,\eta)$ which satisfies all the assumptions of Lemma \ref{L:3} and the boundary condition \eqref{E:T1}. 
We observe that in this case  the restriction to the flow of being non-laminar is equivalent to saying that $\eta=h(\cdot,0)$ is not constant. 
This follows easily from \eqref{R:1} and \eqref{R:2} by means of elliptic maximum principles.
Without restricting the generality, we may assume that
\[
h_q(\cdot,p)\not\equiv0 \qquad\text{for all $p\in(-p_0,0].$}
\]
Indeed, if this is not the case let
\[
\0'=\{(q,p)\in\0\,:\, \text{$h(\cdot,p)$ is not constant}\}.
\]
Then, in $\0\setminus\0'$ the flow is laminar and $h$ is constant on the lower boundary of $\0'.$
Thus, by choosing some other constants $p_0$ and $d$,  we may assume that the height function $h$ associated to our solution 
satisfies \eqref{R:1} in $\0'$ and \eqref{R:2} on the lower boundary of  $\0'.$

Applying Lemma \ref{L:2} we obtain that
\begin{equation}\label{gaucho}
 a_3=a_2=a_1=a_0=0\qquad\text{in $[-p_0,0]$.}
\end{equation}
 We exploit now this relation to arrive  at a contradiction. 
First, from Lemma \ref{L:3} and $a_3\equiv0$, we obtain that
\begin{equation}\label{Pres}
Q(p)=Ap+B 
\end{equation}
for all $p\in[-p_0,0]$ and with a nonzero constant $A$ and some $B\in\R$.
Indeed, if $A=0$, then we infer from \eqref{has} that $h$ has to be laminar in $\0.$ 
Moreover it follows readily from $a_2=0$ that there exists a constant $C_0$ such that  
\begin{equation}\label{beta}
\beta(p)=\frac{A}{2\alpha}\Gamma(p)+\frac{A^2}{\alpha}p+C_0 
\end{equation}
for all $p\in[-p_0,0].$
Lastly, exploiting the fact that $a_1=0,$ we identify  a differential equation for $\Gamma$  
\begin{equation}\label{gamma}
 (c_0\Gamma+c_1)\Gamma'=-(c_2\Gamma+c_3)
\end{equation}
in $[-p_0,0],$ 
the constants $c_i$ being given by the following expressions 
\begin{equation}\label{ci}
c_0:=\frac{A^2}{2\alpha},\quad c_1:=\frac{A^2B}{\alpha}-AC_0,\quad c_2:=\frac{A^3}{\alpha},\quad c_3:=\frac{2A^3B}{\alpha}-\alpha A-2A^2C_0.
\end{equation}
We also need the following  result.
\begin{lemma}\label{L:4}
We have that
\begin{equation}\label{id}
 \left(h+\frac{\beta}{A}\right)^2h_q^2+\left(h+\frac{\beta}{A}+\frac{\alpha}{A^2}\right)^2=-\frac{1}{A^2}\left(\Gamma+\frac{c_3}{c_2}\right)
\end{equation}
 in $\0.$
\end{lemma}
\begin{proof}
Invoking the second relation of \eqref{R:1}, \eqref{has}, and using \eqref{Pres} we find that
\[
(Ah+\beta)^2h_q^2+(Ah+\beta)^2+2\alpha h+2Q+2\Gamma=0
\]
in $\0,$ or equivalently
\[
\left(h+\frac{\beta}{A}\right)^2h_q^2+\left(h+\frac{\beta}{A}+\frac{\alpha}{A^2}\right)^2-\frac{2\alpha\beta}{A^3}-\frac{\alpha^2}{A^4}+\frac{2Q}{A^2}+\frac{2\Gamma}{A^2}=0.
\]
The desired relation follows now from \eqref{beta} and \eqref{ci}.
\end{proof}

Observe from \eqref{id} that $\Gamma+c_3/c_2<0$ for all $p\in(-p_0,0].$
We infer then from \eqref{gamma} and \eqref{ci} that there exists a constant $C_1$ such that 
\begin{equation}\label{gamma1}
 \Gamma+\frac{\alpha^2}{A^2}\ln \left(-\Gamma-\frac{c_3}{c_2}\right)=-2Ap+C_1
\end{equation}
in $[-p_0,0].$
This shows  in fact that $\Gamma+c_3/c_2<0$ on $[-p_0,0].$
In turn, this implies   $c_0\Gamma+c_1\neq0$ for all $\Gamma\in[m,M],$ where $m:=\min_{[-p_0,0]}\Gamma$ and $M:=\max_{[-p_0,0]}\Gamma.$
Letting $f:(m-\e,M+\e)\to\R$ be the function defined by $f(\Gamma):=-(c_2\Gamma+c_3)/(c_0\Gamma+c_1),$
we see that this function is well-defined and real-analytic on $(m-\e,M+\e)$ provided $\e$ is small.
Since $\Gamma'=f(\Gamma),$ we deduce that $\Gamma$ possesses an extension (which is called also $\Gamma$) defined on some interval $(-p_0-\delta,\delta)$ with $\delta>0.$
Additionally, $\Gamma$ inherits the regularity of $f$, that is $\Gamma$ is real-analytic on $(-p_0-\delta,\delta)$.

We define now the function $\wt h:\R\times [-p_0-\delta/2,0]\to\R$ by setting $\wt h(q,p)=h(q,p)$ in $\0,$ 
and 
\[
\wt h(q,p):=h(q,-p_0)+\int_{-p_0}^p\frac{1}{\sqrt{1/h_p^2(q,-p_0)+2\Gamma(-p_0)-2\Gamma(s)}}\, ds 
\]
for $p\in[-p_0-\delta/2,-p_0].$
We note first that $\wt h$ is well-defined.
Indeed, this is a consequence of the fact that $h(q,-p_0)=-d$ and of relation \eqref{has}.
Moreover, $\wt h\in C^2(\R\times[-p_0-\delta/2,0])$ is the solution of the quasilinear elliptic equation
\begin{equation*}
(1+\wt h_q^2)\wt h_{pp}-2\wt h_p\wt h_q\wt h_{pq}+\wt h_p^2\wt h_{qq}-\gamma \wt h_p^3=0\qquad\text{in $\R\times[-p_0-\delta/2,0]$},
\end{equation*}
 with $\wt h$ satisfying \eqref{COND1'} in $\R\times[-p_0-\delta/2,0].$
We enhance that condition \eqref{COND1'} guarantees that this quasilinear  equation  is  uniformly elliptic. 
Since $\gamma$ is   real-analytic and all three equations of the system depend analytically on $\wt h$, the regularity results in \cite{GT01, Mo66} imply that $\wt h$ is real-analytic in the strip $\R\times(-p_0-\delta/2,0)$.
But, since $\wt h_q\equiv0$ in $\R\times[-p_0-\delta/2,-p_0],$ the principle of analytic continuation implies
the flow must be laminar, in contradiction with our assumption.
 This finishes the proof of Theorem \ref{MT:1}. 

\section{Proof of Theorem \ref{MT:2}}\label{S4}
In this section we consider a  non-laminar solution  $(u,v,P,\eta)$  which satisfies all the assumptions of Lemma \ref{L:3} and the far field condition \eqref{E:T2}, that is a solution of the deep water wave problem. 
Furthermore, we let $h$ be the corresponding height function.
From the proof of Theorem \ref{MT:1} we deduce that it is possible to choose $p_0\geq0$ such that $h_q(\cdot,p)\not\equiv0$ for all $p\leq-p_0.$ 

We shall restrict our attention first to the flow in $\0':=\R\times(-\infty, -p_0).$
By the results of the previous section, the relations \eqref{Pres}, \eqref{beta}, \eqref{gamma} (and subsequently \eqref{gamma1})  are satisfied on $(-\infty,-p_0],$ and the identity \eqref{id} takes place in all $\0'.$
Particularly, $\Gamma+c_3/c_2<0$ in $(-\infty,-p_0].$
This implies that 
\begin{equation}\label{A}
 A<0.
\end{equation}
Indeed, if $A>0$, then we infer from \eqref{gamma1} that $\Gamma(p)\to\infty$ as $p\to-\infty$, which  contradicts the fact that $\G$ is bounded from above.

Our previous analysis  allows us to prove now that the flow possesses some of the characteristics of Gerstner's solution for deep water waves in the $f-$plane approximation.

\begin{lemma}[The flow is rotational] \label{L:5}
There exists $\e>0$ such that $\gamma$ is real-analytic on $(-\infty, -p_0+\e).$
Moreover, $\gamma'<0$ and $\gamma\to_{p\to-\infty}0.$ 
\end{lemma}
\begin{proof} Since $\Gamma<-c_3/c_2$ in $(-\infty,-p_0]$, we infer  from the equation \eqref{gamma} that $\Gamma(p)\neq -c_1/c_0$ for all $p\in(-\infty,-p_0]$.
It follows readily from  \eqref{ci} that $-c_1/c_0<-c_3/c_2$, so that either $\Gamma<-c_1/c_0$ or $-c_1/c_0<\Gamma<-c_3/c_2.$
 By \eqref{gamma1} and \eqref{A}, we can exclude the alternative $\Gamma<-c_1/c_0$, so that 
 \begin{equation}\label{G}
 -c_1/c_0<\Gamma<-c_3/c_2\qquad\text{for all $p\in(-\infty,-p_0].$}
\end{equation}
Particularly, the same arguments as in the previous section imply that $\Gamma$ is real-analytic on $(-\infty,-p_0+\e)$, provided $\e>0$ is small, and since by  \eqref{gamma} 
\[
\gamma=\Gamma'=-\frac{c_2}{c_0}\frac{\Gamma+c_3/c_2}{\Gamma+c_1/c_0}=-2A\frac{\Gamma+c_3/c_2}{\Gamma+c_1/c_0},
\]
we use \eqref{A} and \eqref{G} to obtain that $\gamma<0$ and $\gamma\to_{p\to-\infty} 0 $ (or equivalently $\Gamma(p)\to_{p\to-\infty}-c_3/c_2.$)
Differentiating the last equation once we see that $\gamma'<0.$ 
This completes our argument. 
\end{proof}

Since $\gamma$ is real-analytic on $(-\infty,-p_0+\e),$ the regularity results in \cite{Matxx} imply that all the streamlines $h(\cdot,p),$ $p<-p_0+\e$,  are real-analytic functions.
Even more, using   regularity results for quasilinear elliptic equations \cite{Mo66}, we know by \eqref{COND1'} that $h$ is real-analytic in  $\0'$. 
Let now 
\[
 K(p):=-\frac{1}{A}\left(-\Gamma-\frac{c_3}{c_2}\right)^{1/2}\qquad\text{for $p\leq -p_0$}.
\]
With this notation, equation \eqref{id} becomes 
\begin{equation}\label{id'}
 \left(h+\frac{\beta}{A}\right)^2h_q^2+\left(h+\frac{\beta}{A}+\frac{\alpha}{A^2}\right)^2=K^2\qquad\text{in $\0'.$}
\end{equation}
Furthermore, we define
\begin{equation}\label{par}
  (x(s,p),z(s,p)):=\left(\frac{\alpha}{A}s-K(p)\sin(As),-\frac{\alpha}{A^2}-\frac{\beta(p)}{A}+K(p)\cos(As)\right)
\end{equation}
for all $s\in\R$ and $p\leq -p_0.$ 
Fixing $p\leq -p_0$, we have that
\[
 x_s=\displaystyle\frac{\alpha}{A}-AK(p)\cos(As), \qquad s\in\R,
\]
and since $K(p)\to_{p\to-\infty }0$ we may take $p_0$ large to guarantee  that $ x(\cdot,p):\R\to\R$ is a diffeomorphism for all $p\leq-p_0.$
This fact allows us to define the map $ s:\R\times(-\infty,-p_0]\to\R$  by the relation $ x( s(q,p),p)=q$ for all $(q,p)\in\R\times(-\infty,-p_0]$.
Note that the implicit function theorem ensures that $ s$ is real-analytic as well.
Finally, let $\wt  h:\R\times(-\infty,-p_0]\to\R$ be defined by
\[
\wt h(q,p):= z( s(q,p),p)\qquad\text{for $(q,p)\in\R\times(-\infty,-p_0]$.}
\]
Then,  using the chain rule, we find that $\wt h$ satisfies the following relation
\begin{align*}
 &\left(\wt h+\frac{\beta}{A}\right)^2\wt h_q^2+\left(\wt h+\frac{\beta}{A}+\frac{\alpha}{A^2}\right)^2=\left(-\frac{\alpha}{A^2}+K\cos(A s)\right)^2\left(\frac{z_s}{ x_s}\right)^2+K^2\sin^2(A s)\\
&\qquad=\left(\frac{\alpha}{A^2}-K\cos(A s)\right)^2\frac{A^2K^2\cos^2(A s)}{\left(\alpha/A-AK\cos(A s)\right)^2}+K^2\sin^2(As)\\
&\qquad=K^2.
\end{align*}
Since $ h(\cdot,p)$ and $\wt h(\cdot,p)$ are both real-analytic solutions of \eqref{id'}, by the theorem of Picard-Lindel\"of we find a function $\delta=\delta(p)$, $p\leq -p_0,$
such that
\begin{equation}\label{identify}
h(q,p)=\wt h(q+\delta(p),p) \qquad\text{ in $\R\times(-\infty,-p_0]$.}
\end{equation}
Using the implicit function theorem and \eqref{id},  we may choose $\delta$ to be a real-analytic function. 
Furthermore, since $h(q,p)\to-\infty$ when $p\to-\infty,$ we infer from \eqref{beta}, \eqref{par}, and \eqref{identify} that
\begin{equation}\label{alpha}
 \alpha>0.
\end{equation}

Given  $s\in\R$ and $p\leq-p_0$, we let $q\in\R$ satisfy $ s= s(q+\delta(p),p).$
Then  $q= x( s,p)-\delta(p)$, 
\[
h(q,p)=\wt h(q+\delta(p),p)= z( s(q+\delta(p),p),p)=z(s,p),
\]
and therefore we have
\begin{align*}
(u-c)(x( s,p)-\delta(p),   z(s,p))=&(u-c)(q, h(q,p))=-\frac{1}{h_p}(q,p)=-Ah(q,p)-\beta(p)\\
=&-Az(s,p)-\beta(p)=\frac{\alpha}{A}-AK(p)\cos(As)\\
=&\frac{\p}{\p s}(x(s,p)-\delta(p)),
\end{align*}
respectively
\begin{align*}
w(x( s,p)-\delta(p),   z( s,p))=&w(q, h(q,p))=-\frac{h_q}{h_p}(q,p)=-\wt h(q+\delta(p),p)(Az(s,p)+\beta(p))\\
=&\frac{-AK(p)\sin(As)}{\alpha/A-AK(p)\cos(As)}\left(\frac{\alpha}{A}-AK(p)\cos(As)\right)\\
=&\frac{\p z}{\p s}(s,p).
\end{align*}

This shows that the  path of an arbitrary particle in the  steady flow  beneath $y=h(\cdot,-p_0)$ is described by the curve
\[
 t\mapsto (\wt x(t+s,p),\wt z(t+s,p)):=(x(t+s,p)-\delta(p),z(t+s,p)),
\] 
whereby $(\wt x(s,p),\wt z(s,p))$, with $(s,p)\in\R\times(-\infty,-p_0]$ is the initial position of the particle.
Back to the original reference frame $(X,Z),$ the position of the particles at any time  is described by the mapping
\[
[0,\infty)\times \R\times(-\infty,-p_0]\ni (t,s,p)\mapsto (\wt X(t,s,p),\wt Z(t,s,p)):=(\wt x(t+s,p)+ct,\wt z(t+s,p)). 
\]
Indeed, one can  easily check that $(\wt X(\cdot, s,p),\wt Z(\cdot, s,p ))$ are   solutions of the non-autonomous system
\[
 \left\{
\begin{array}{llll}
 X_t&=u(t,X,Z)=u(X-ct, Z),\\
Z_t&=w(t,X,Z)=w(X-ct,Z)
\end{array}
\right.
\]
for each $s\in\R$ and $p\leq-p_0.$

In the remainder of this section we show that the equations for the particle paths can be brought in the form \eqref{Lag}.
To this end, we introduce    new variables 
\[a:=\frac{\alpha}{A}s\qquad\text{and}\qquad b:=T(p)=\displaystyle\frac{\ln\left(\frac{A^2}{\alpha}K(p)\right)}{\frac{A^2}{\alpha}},\]
with $a\in\R$ and $b\leq b_0:=H(-p_0).$
That the map $T:(-\infty,-p_0]\to (-\infty,b_0]$ is a diffeomorphism is a consequence of the relation
\[
T'(p)=\frac{\alpha}{A^2}\frac{K'(p)}{K(p)}=\frac{\alpha}{2A^2}\frac{2K(p)K'(p)}{K^2(p)}=-\frac{\alpha}{2A^4}\frac{\gamma(p)}{K^2(p)}>0\qquad\text{for $p\leq -p_0,$}
\]
cf. \eqref{alpha} and Lemma \ref{L:5}. 
Letting the positive constants $k$ and $m $ be defined by
\[
k:=\frac{A^2}{\alpha}\qquad\text{and}\qquad m:=-\frac{\alpha}{A},
\] 
the  paths of the particles located initially beneath the curve $(\wt X(0,\cdot,-p_0), \wt Z(0,\cdot,-p_0))$, which we may regard as being the surface of the wave, 
are described in the  reference frame located on the Earth's surface by the mappings
\begin{equation}\label{Lag1}
\left\{ 
\begin{array}{llll}
&\displaystyle X(t,a,b):=-\wt\delta(b)+a+\left(c-m\right)t-\frac{e^{kb}}{k}\sin(k(a-mt)),\\[1ex]
&\displaystyle Z(t,a,b):=h_0+b+\frac{e^{kb}}{k}\cos(k(a-mt)),
\end{array}
\right.
\end{equation}
 with $a\in\R$, $b\leq b_0.$ 
Hereby, we have defined  $\wt \delta:=\delta\circ T^{-1}$ and $h_0$ is a suitable constant.
The second equation  of \eqref{Lag1} follows from the relation
\[
 (\beta\circ T^{-1})'(b)=\frac{A^2}{\alpha}\left(\beta'\frac{K}{K'}\right)\circ H^{-1}=\frac{2A^2}{\alpha}\left(\beta'\frac{K^2}{(K^2)'}\right)\circ H^{-1}=-A\qquad\text{for all $b\leq b_0,$} 
\]
which is a consequence  of the fact that  \eqref{gamma1} may be written in the equivalent form
\[
\beta=-\frac{\alpha}{2A}\ln(A^2K^2)+\frac{AC_1}{2\alpha}-C_0\qquad\text{in $(-\infty,-p_0]$}.
\]
That the flow generated by \eqref{Lag1} is incompressible reduces to showing that 
\[
X_{ta}Z_b-X_{tb}Z_a-Z_{ta}X_b+Z_{tb}X_a=0\qquad\text{in $[0,\infty)\times\R\times(-\infty,-p_0],$}
\]
cf. Lemma 3.4 in  \cite{AM12xx}.
However, the latter condition is equivalent to $\wt\delta'=0,$ meaning that $\delta$ is a constant function.
Translating the original solution $h$ suitably, we may assume without restricting the generality that $\delta\equiv0.$ 
Back to the coordinates $(s,p)$ we have that 
\[
(u-c)(x(s,p),z(s,p))=\frac{\alpha}{A}-AK(p)\sin(As)\to\frac{\alpha}{A},
\] 
when $p\to-\infty,$ uniformly in $s$.
But then, by \eqref{E:T2}, we identify $m$ with the wave speed $m=c,$ 
so that the trajectories of the particles are described by
\begin{equation}\label{Lag2}
\left\{ 
\begin{array}{llll}
&\displaystyle X(t,a,b)=a-\frac{e^{kb}}{k}\sin(k(a-ct)),\\[1ex]
&\displaystyle Z(t,a,b)=h_0+b+\frac{e^{kb}}{k}\cos(k(a-ct))
\end{array}
\right.
\end{equation}
 for all  $a\in\R$, $b\leq b_0,$ and $t\geq0.$

Since $h$ is real-analytic in the interior of the set $\{(q,p)\,:\, h_q(\cdot,p)\not\equiv0\}$, we conclude that $\0'=\0$, and that there exists a $b_0\leq0$ such that our original solution  is described by the equations \eqref{Lag2} 
in $\0_\eta.$
This finishes the proof.

\vspace{0.5cm}
\hspace{-0.5cm}{\large \bf Acknowledgement}\\[2ex]
A.-V. Matioc was supported by the FWF Project I544 --N13 ``Lagrangian kinematics of water waves'' of the Austrian Science Fund.\\

\bibliographystyle{abbrv}
\bibliography{AMBM.bib}
\end{document}